\documentclass[11pt, a4paper]{amsart}
\usepackage{amsmath}
\usepackage{amssymb}
\usepackage{color}

\newtheorem{prop}{Proposition}[section]
\newtheorem{teo}{Theorem}[section]
\newtheorem{lema}{Lemma}[section]

\theoremstyle{definition}

\def\ep{\varepsilon}

\begin{document}
\title[Near field limit for the PME on the half-line]{Near field asymptotic behavior for the porous medium equation on the half-line}

\author[Cort\'{a}zar, Elgueta, Quir\'{o}s \and Wolanski]{C. Cort\'{a}zar,  F. Quir\'{o}s \and N. Wolanski}

\address{Carmen Cort\'{a}zar\hfill\break\indent
Departamento  de Matem\'{a}tica, Pontificia Universidad Cat\'{o}lica
de Chile \hfill\break\indent Santiago, Chile.} \email{{\tt
ccortaza@mat.puc.cl} }

\address{Fernando Quir\'{o}s\hfill\break\indent
Departamento  de Matem\'{a}ticas, Universidad Aut\'{o}noma de Madrid
\hfill\break\indent 28049-Madrid, Spain.} \email{{\tt
fernando.quiros@uam.es} }

\address{Noemi Wolanski \hfill\break\indent
Departamento  de Matem\'{a}tica, FCEyN,  UBA,
\hfill\break \indent and
IMAS, CONICET, \hfill\break\indent Ciudad Universitaria, Pab. I,\hfill\break\indent
(1428) Buenos Aires, Argentina.} \email{{\tt wolanski@dm.uba.ar} }

\thanks{C. Cort\'azar supported by  FONDECYT grant 1150028 (Chile). F. Quir\'os supported by
project MTM2014-53037-P (Spain). N. Wolanski supported by
CONICET PIP625, Res. 960/12, ANPCyT PICT-2012-0153, UBACYT X117 and MathAmSud 13MATH03 (Argentina).}



\date{}

\begin{abstract}
Kamin and V\'azquez proved in 1991 that solutions to the Cauchy-Dirichlet problem for the porous medium equation $u_t=(u^m)_{xx}$ on the half line with zero boundary data and nonnegative compactly supported integrable initial data behave for large times as a dipole type solution to the equation having the same first moment as the initial data, with an error which is $o(t^{-1/m})$. However, on sets of the form $0<x<g(t)$, with $g(t)=o(t^{1/(2m)})$ as $t\to\infty$, in the so called near field, the dipole solution is $o(t^{-1/m})$, and their result does not give neither the right rate of decay of the solution, nor a nontrivial asymptotic profile. In this paper we will show that the error is $o\big(t^{-(2m+1)/(2m^2)}(1+x)^{1/m}\big)$. This allows in particular to obtain a nontrivial  asymptotic profile in the near field limit, which is a multiple of $x^{1/m}$, thus improving in this scale the results of Kamin and V\'azquez.
\end{abstract}

\maketitle

\section{Introduction}
\label{sect-Introduction} \setcounter{equation}{0}

This paper is concerned with the large time behavior of solutions to the porous medium equation (PME in what follows) in the half line,
\begin{equation}
\label{eq:main}
u_t=(u^m)_{xx},\quad (x,t)\in\mathbb{R}_+\times\mathbb{R}_+,\qquad u(x,0)=u_0(x),\quad x\in\mathbb{R}_+,
\end{equation}
with $m>1$, and nonnegative, and compactly supported integrable initial data. This problem, which models the flow of a fluid in a porous medium, has a unique weak solution; see~\cite{Vazquez-book}.
The asymptotic behavior was first studied by Kamin and V\'azquez in \cite{Kamin-Vazquez-1991}, and depends heavily on the fact that solutions to~\eqref{eq:main} preserve the first moment along the evolution, $\int_0^\infty x u(x,t)\,dx=\textrm{constant}$ for all $t>0$. Indeed, the behavior is given in terms of the so called \emph{dipole solution} of the PME with first moment $M=\int_0^\infty xu_0(x)\,dx$,
\begin{equation}
\label{eq:selfsimilar.form.dipole}
D_M(x,t)=t^{-\alpha}F_M(\xi), \qquad\xi=x/t^\beta,\quad \alpha=\frac1m,\ \beta=\frac1{2m}.
\end{equation}
Note that this solution has a self-similar structure, and that, due to the choice of the similarity exponents $\alpha$ and $\beta$, its first moment is constant in time. The precise result in~\cite{Kamin-Vazquez-1991} states that
\begin{equation}
\label{eq:KV.result}
\lim_{t\to\infty}t^{\alpha}\sup_{x\in\mathbb{R}_+}|u(x,t)-D_M(x,t)|=0,\qquad M=\int_0^\infty xu_0(x)\,dx.
\end{equation}
In order for $D_M$ to be a weak  solution of the equation on the half-line, with $D_M(0,t)=0$, the profile $F_M$ has to solve, in a weak sense,
\begin{equation}
\label{eq:profile}
(F_M^m)''(\xi)+\beta\xi F_M'(\xi)+\alpha F_M(\xi)=0,\quad\xi\in\mathbb{R}_+, \qquad F_M(0)=0,
\end{equation}
while the condition on the value of the first moment imposes
$$
\int_0^\infty \xi F_M(\xi)\,d\xi=M.
$$
A simple scaling argument shows that $F_M(\xi)=M^{1/m}F_1(\xi/M^{(m-1)/(2m)})$. It turns out that there is a unique bounded profile corresponding to $M=1$, which is moreover explicit,
\begin{equation}\label{eq-F-moment}
F_1(\xi)=\xi^{\frac1m}\Big(C_m-\kappa_m\,\xi^{\frac{m+1}m}\Big)_+^{\frac1{m-1}},
\end{equation}
with constants $C_m$ and $\kappa_m$ given by
\begin{equation}
\label{eq:constants}
\kappa_m=\frac{m-1}{2m(m+1)},\qquad C_m=\left(\frac{\kappa_m^{\frac{2m+1}{m+1}}}{\int_0^1 s^{\frac{m+1}{m}}(1-s^{\frac{m+1}{m}})^{\frac1{m-1}}\,ds}\right)^{(m^2-1)/(2m^2)};
\end{equation}
see~\cite{Barenblatt-Zeldovich-1957,Gilding-Peletier-1976,Gilding-Peletier-1977}.
Note that $F_1$ has a compact support, $[0,\xi_1]$, where $\xi_1=(C_m/\kappa_m)^{m/(m+1)}$. Thus, $F_M$ has compact support $[0,\xi_M]$, with
$$
\xi_M=\xi_1 M^{(m-1)/(2m)}.
$$
Let us remark that $\lim_{t\to0^+}\int_0^\infty D_M(x,t)\varphi(x)\,dx=M\varphi'(0)$. In other words, the antisymmetric extension of $D_M$, $\overline D_M$, satisfies $\overline D_M(\cdot,t)\to -2M\delta'$, where $\delta'$ is the distributional derivative of the delta function. This is called in physics an elementary dipole. Hence the name dipole solution of the PME for $D_M$.

\medskip

\noindent\emph{Remarks. } (a) The result in \cite{Kamin-Vazquez-1991} states that solutions to the signed PME in the whole real line
\begin{equation}
\label{eq:signed.PME}
u_t=(|u|^{m-1}u)_{xx},\qquad (x,t)\in\mathbb{R}\times\mathbb{R}_+
\end{equation}
with an integrable and compactly supported initial data having zero mass and a nontrivial first moment $\int_{\mathbb{R}}xu(x,0)\,dx=P$ converge to $\overline D_{P/2}$, with an error which is $o(t^{-\alpha})$.  Solutions  to~\eqref{eq:main} clearly fall within this framework, since they coincide with the restriction to the half line of the solution to~\eqref{eq:signed.PME} having as initial datum the antisymmetric extension of $u_0$.

\noindent (b) The proof in \cite{Kamin-Vazquez-1991} uses that $v(x,t)=\int_{-\infty}^x u(y,t)\,dy$ is a solution to the $p$-laplacian evolution equation with $p=m+1$,
\begin{equation}
\label{eq:p.laplacian}
v_t=(|v_x|^{m-1}v_x)_x,\qquad (x,t)\in\mathbb{R}\times\mathbb{R}_+.
\end{equation}
Convergence for $u$ is deduced from the convergence of $v$ and its derivatives. In particular, $D_M=\partial_x B_{-2M}$, where $B_K$ denotes the source type solution to~\eqref{eq:p.laplacian}, that has $K\delta$ as initial datum.

\medskip

The limit~\eqref{eq:KV.result} gives the first nontrivial term in the asymptotic expansion of $u$ for $x=O(t^\beta)$, in the so called \emph{far field limit}. However, since $F_M(0)=0$, in the \emph{near field}, $x=o(t^\beta)$, it only says that $u$ is $o(t^{-\alpha})$. The aim of this paper is to improve the result of Kamin and V\'azquez in the near field, giving a sharp decay rate, which is faster than that in the far field, and a nontrivial asymptotic profile, which turns out to be a multiple of $x^{1/m}$. The precise result reads as follows.
\begin{teo} Let $u$ be the unique weak solution to~\eqref{eq:main}. Then
\begin{equation}
\label{eq:main.result}
\lim_{t\to\infty} t^{\alpha+\frac\beta m}\sup_{x\in\mathbb{R}_+}\frac{\big|u(x,t)-D_M(x,t)\big|}{(1+x)^{\frac1 m}}=0,
\end{equation}
where $D_M$ is the unique dipole solution to the PME with first moment $M=\int_0^\infty xu_0(x)\,dx$.
\end{teo}
In particular, on compact sets, $x\in [0,K]$,  $u$ is $O(t^{-(\alpha+\frac\beta m)})$, and $t^{\alpha+\frac\beta m}u(x,t)$ converges to
$M^{2/(m+1)} C_m^{1/(m-1)}x^{1/m}$, while
$$
\sup_{0\le x\le g(t)}u(x,t)=O\Big(t^{\alpha+\frac\beta m}/(g(t))^{1/m}\Big)
$$
if $g(t)\to\infty$, $g(t)=o(t^\beta)$.

We already know that the result is true in the far-field scale, $\xi_1\le x/t^\beta\le \xi_2$, $0<\xi_1\le\xi_2<\infty$; see~\eqref{eq:KV.result}.
If $x\ge g(t)t^\beta$, with $g(t)\to\infty$ as $t\to\infty$ (the very far field scale), formula \eqref{eq:KV.result} yields a better result. Nevertheless, since $D_M(x,t)=0$ for $x\ge \xi_M t^\beta$, it only says that the solution is $o(t^{-\alpha})$ there. However, by using some asymptotic formulas from \cite{Esteban-Vazquez-1988}, Kamin and V\'azquez prove that $s(t)=\sup\{x:u(x,t)>0\}$ satisfies
$$
s(t)=\xi_M t^\beta+o(1),\quad s(t)\ge \xi_M t^\beta,
$$
which gives a complete characterization of the asymptotic behavior for $x>\xi_M t^\beta$. Hence, it only remains to check what happens in the near-field. This is done through a matching argument with the outer behavior, which is based on a clever choice of sub and supersolutions. Comparison is performed in sets of the form
$0< x <\delta t^\beta$, for some small $\delta$. The ordering in the outer boundary comes from the outer behavior which was  already known from the analysis in \cite{Kamin-Vazquez-1991}. We devote Section~\ref{sect-supersolution} to obtain the upper limit and Section~\ref{sect-subsolution} to get the lower one.

A similar analysis for the linear heat equation has been recently performed in \cite{CEQW2}. However, in that case linearity made things easier, since a representation formula for the solution in terms of the initial datum was available. That paper also considers a nonlocal version of the heat equation.

\section{Control from above}
\label{sect-supersolution} \setcounter{equation}{0}

The purpose of this section is to prove the \lq\lq upper'' part of~\eqref{eq:main.result}. To this aim we will construct a supersolution $V$ approaching $D_M$ with the right rate as $t$ goes to infinity.  We only need the function $V$ to  be a supersolution in sets of the form
$$
A_{\delta,T}=\{(x,t)\,,\,t\ge T,\ 0<x<\delta t^\beta\}
$$
for $T>0$ big and $\delta>0$ small. Our candidate is
\begin{equation}
\label{eq:V}
V(x,t)= k(t)t^{-\alpha}F_M\Big(\frac{x+a}{t^{\beta}}\Big), \quad a>0,
\end{equation}
for some function $k$ satisfying $k(t)\searrow 1$ as $t\to\infty$. It will turn out that a good choice for $k$ is given by the solution to
\begin{equation}
\label{eq:k}
tk'(t)=-\alpha{(k^m(t)-k(t))},\quad t>T,\qquad k(T)=k_0>1.
\end{equation}
Note that $k(t)$ is well defined and that it is a monotone decreasing function of time.

We start by proving that $V$ is a supersolution to the PME in $A_{\delta,T}$ if $\delta$ is small and $T$ is big.

\begin{lema}
\label{lem:supersolution}
Let $m>1$ and $M>0$. There exist values $\bar\delta>0$ and $\overline T>0$  depending only on $M$ and $m$ such that for all $a\in(0,1)$, $T\ge\overline T$, and $k_0>1$  the function $V$ given by~\eqref{eq:V}--\eqref{eq:k} satisfies
$$
V_t-(V^m)_{xx}\ge 0 \quad\text{in }A_{\delta,T} \text{ for all }\delta\in(0,\bar\delta).
$$
\end{lema}
\begin{proof}
Let $\xi=\frac {x+a}{t^{\beta}}$. Since $m\alpha+2\beta=\alpha+1$, a straightforward computation combined with~\eqref{eq:profile} shows that
\[
\big(V_t-(V^m)_{xx}\big)(x,t)= t^{-\alpha-1}\big(tk'(t)F_M(\xi)+(k(t)-k^m(t)) (F_M^m)''(\xi)\big).
\]
Thus, if we choose $k$ satisfying~\eqref{eq:k}, we get
\[
\big(V_t-(V^m)_{xx}\big)(x,t)= t^{-\alpha-1}\big(k^m(t)-k(t)\big)\big(-\alpha F_M(\xi)-(F_M^m)''(\xi)\big),\quad x\in\mathbb{R}_+,\ t\ge T.
\]
Now we observe that there is a value $\bar \xi\in(0,\xi_1)$ such that $F_1'(\xi)>0$ for $\xi\in(0,\bar \xi)$. Therefore,
\begin{equation}
\label{eq:positivity.derivative}
F_M'(\xi)>0, \quad \xi\in\big(0,\bar\xi M^{(m-1)/(2m)}\big).
\end{equation}
On the other hand, if we take $\bar\delta<\bar\xi M^{(m-1)/(2m)}/2$, and then $\overline T=(1/\bar\delta)^{1/\beta}$, for any $\delta\in(0,\bar\delta)$ and $T\ge\overline T$ we get
$$
0< \xi=\frac {x+a}{t^{\beta}}\le 2\bar\delta<\bar\xi M^{(m-1)/(2m)},\quad (x,t)\in A_{\delta,T}.
$$
Hence,  $-(F_M^m)''(\xi)-\alpha F_M(\xi)=\beta\xi F_M'(\xi)>0 $ if $(x,t)\in A_{\delta,T}$, and the result follows, since $k(t)>1$ for all times.
\end{proof}

We now arrive at the matching part of the result where, using the behavior in the far field scale, we obtain an upper bound in sets of the form $A_{\delta, T}$ for $\delta$ small and $T$ large.

\begin{lema}
\label{lem:control-from-above} Let $u$ be the unique weak solution to \eqref{eq:main}, $M=\int_0^\infty xu_0(x)\,dx$,  and $\bar\delta$ and $\overline T$ as in Lemma~\ref{lem:supersolution}. For every $\ep>0$ there exists a value $T_\ep\ge \overline T$  such that for all $a\in(0,1)$, and $T\ge T_\ep$ there is a value $k_0\ge1$ such that the function $V$ given by~\eqref{eq:V}--\eqref{eq:k} satisfies
\begin{equation}
\label{eq:upper.estimate}
u(x,t)\le \Big(1+C_\delta\ep\Big)V(x,t),\quad (x,t)\in A_{\delta,T}, \qquad C_\delta=1/F_M(\delta),\qquad \delta\in(0,\bar\delta).
\end{equation}
\end{lema}
\begin{proof}
Formula~\eqref{eq:KV.result} implies that  for every $\ep>0$ there exists $T_\ep$, which may be assumed to be larger than $\overline T$, such that
\begin{equation}
\label{eq:outer.epsilon.approx}
t^{\alpha}\big|u(x,t)-D_M(x,t)\big|\le \ep,\quad x\in\mathbb{R}_+,\ t\ge T_\ep.
\end{equation}
For any given $T\ge T_\ep$ and some big enough $k_0>1$ to be determined below, we define $V$ by~\eqref{eq:V}--\eqref{eq:k}.

Since $k(t)>1$, \eqref{eq:positivity.derivative} implies that $D_M(x,t)\le V(x,t)$ in $A_{\delta, \overline T}$, and hence in $A_{\delta, T}$. Therefore,
\[
\begin{aligned}
t^{\alpha}&\big(u(x,t)-V(x,t)\big)
\le
t^{\alpha}|u(x,t)-D_M(x,t)|
 + t^{\alpha}\big(D_M(x,t)-V(x,t)\big)
 \le \ep,\quad (x,t)\in A_{\delta,T}.
 \end{aligned}
 \]
On the other hand, for $x=\delta t^\beta$  and $t\ge T$,
\[
t^{\alpha}V(x,t)\ge t^{\alpha}D_M(x,t)=F_M\Big(\frac x{t^\beta}\Big)=F_M(\delta),
\]
and we conclude that
\[
t^{\alpha}\big(u(x,t)-V(x,t)\big)\le  t^{\alpha}C_\delta\ep V(x,t).
\]
Thus,
\[
u(x,t)\le \underbrace{(1+C_\delta\ep) V(x,t)}_{W(x,t)}\quad\mbox{if } x=\delta t^\beta,\ t\ge T.
\]

Since solutions to~\eqref{eq:signed.PME} with integrable initial data are bounded for $t\ge\tau>0$, see~\cite{Vazquez-book}, formula~\eqref{eq:KV.result} implies that there is a constant $C_0>0$ such that $u(x,t)\le C_0 t^{-\alpha}$ for $t\ge \overline T$.
Thus, using the monotonicity property~\eqref{eq:positivity.derivative}, we get
\[
V(x, T)=k( T)T^{-\alpha}F_M\Big(\frac{x+a}{ T^{\beta}}\Big)
\ge k_0F_M\Big(\frac{a}{ T^{\beta}}\Big) T^{-\alpha}\ge C_0 T^{-\alpha}\ge u(x, T)\quad\mbox{for }0<x<\delta  T^\beta
\]
if $k_0>\max\{C_0/F_M\big(\frac{a}{ T^{\beta}}\big),1\}$. Therefore, with that choice of $k_0$,
\[
u(x, T)\le V(x, T)\le W(x, T)\quad\mbox{if }0<x<\delta  T^\beta.
\]

We now observe that $W$ is a supersolution to the PME in $A_{\delta,T}$. Indeed, in that set we have $(V^m)_{xx}(x,t)=-k(t)t^{-(\alpha+1)}\big(\alpha F_M(\xi)+\beta\xi F_M'(\xi)\big)<0$, and hence Lemma~\ref{lem:control-from-above} implies that
\[\begin{aligned}
W_t-(W^m)_{xx}&= (1+C_\delta\ep)V_t-(1+C_\delta\ep)^m(V^m)_{xx}\\
&= (1+C_\delta\ep)\big(V_t-(V^m)_{xx})-
\big((1+C_\delta\ep)^m-(1+C_\delta\ep)\big)(V^m)_{xx}\ge 0.
\end{aligned}
\]

We finally notice that $W(0,t)>0$ for all $t> T$. Therefore, comparison yields~\eqref{eq:upper.estimate}.
\end{proof}

The third ingredient, that we prove next, is that $V$ and $D_M$ are $\ep$-close in sets of the form $A_{\delta,T}$ for large times, even when the difference is multiplied by $t^{\alpha+\frac\beta m}/(1+x)^{\frac1 m}$, if the parameter $a$ in the definition of $V$ is $O(\ep^m)$.

\begin{lema}\label{lema-D-V} Let $m>1$, $M>0$, and $\ep>0$, and let $\bar\delta$ and $\overline T$ be as in Lemma~\ref{lem:supersolution}. There exist values $\hat\delta\in(0,\bar\delta)$,  $\widehat T\ge \overline T$ independent of $\ep$, and $a_\ep\in(0,1]$,  such that for all $\delta\in(0,\hat\delta)$,  $T\ge\widehat T$ and $a\in(0,a_\ep)$ the function $V$ given by~\eqref{eq:V}--\eqref{eq:k} satisfies
\begin{equation}
\label{eq:epsilon.approximate}
\frac{t^{\alpha+\frac\beta m}}{(1+x)^{\frac1 m}}\big|V(x,t)-D_M(x,t)\big|<\ep\quad\mbox{in }A_{\delta,\widehat T_\ep}\text{ for some }\widehat T_\ep\ge T.
\end{equation}
\end{lema}
\begin{proof}
There holds that
\[
\begin{aligned}
\frac{t^{\alpha+\frac\beta m}}{(1+x)^{\frac1 m}}\big|D_M(x,t)-V(x,t)\big|=&\underbrace{\frac{t^{\frac\beta m}}{(1+x)^{\frac1 m}}\big|F_M\Big(\frac{x+a}{t^\beta}\Big)-F_M\Big(\frac{x}{t^\beta}\Big)\big|}_{\textrm{I}}\\
&+\underbrace{\frac{t^{\frac\beta m}}{(1+x)^{\frac1 m}}F_M\Big(\frac{x+a}{t^\beta}\Big)\big|k(t)-1\big|}_{\textrm{II}}.
\end{aligned}
\]
In order to estimate I we notice that there exist constants $\hat\xi\in(0,\bar\xi)$ and $K>0$ such that $\xi F_1'(\xi)\le K\xi^{\frac 1m}$ for $\xi\in(0,\hat\xi)$. Thus, if we take $\hat\delta<\hat\xi M^{(m-1)/(2m)}/2$, and then $\widehat T=(1/\hat\delta)^{1/\beta}$, for any $\delta\in(0,\hat\delta)$ and $T\ge\widehat T$, we get
\[\begin{aligned}
\textrm{I}&=
\frac{t^{\frac\beta m}}{(1+x)^{\frac1 m}}\int_0^1F_M'\Big(\frac{x+sa}{t^\beta}\Big)\frac a{t^\beta}\,ds=
\frac{t^{\frac\beta m}}{(1+x)^{\frac1 m}}\int_0^1F_M'\Big(\frac{x+sa}{t^\beta}\Big)\frac {x+sa}{t^\beta}\frac a{x+sa}\,ds\\
&\le K \int_0^1{(x+sa)^{\frac1m-1}} a\,ds\le mK a^{\frac1m}\quad\text{in }A_{\delta,T}.
\end{aligned}\]
Therefore, $\textrm{I}<\ep/2$ if $a<a_\ep:=\min\{(\ep/(2mK))^m,1\}$.

As for the other term, we will use that
\begin{equation}
\label{eq:estimate.profile.above}
F_M(\xi)\le C_m^{\frac1{m-1}}M^{\frac{m+1}{2m^2}}\xi^{\frac1m},\quad\xi\in\mathbb{R}_+;
\end{equation}
see formula~\eqref{eq:profile}. Therefore, taking into account that $a<1$, we obtain
\[\begin{aligned}
\textrm{II}\le C_m^{\frac1{m-1}}M^{\frac{m+1}{2m^2}}\Big(\frac{x+a}{x+1}\Big)^{\frac1m}\big|k(t)-1\big|\le
C_m^{\frac1{m-1}}M^{\frac{m+1}{2m^2}}\big|k(t)-1\big|<\ep/2
\end{aligned}\]
if $t\ge \widehat T_\ep$ for some $\widehat T_\ep\ge T$, since $k(t)\to 1$ as $t\to\infty$.
\end{proof}

We finally arrive at the main result of this section, the upper limit.
\begin{prop}
Let $u$ be the unique weak solution to~\eqref{eq:main}, and $D_M$ be the unique dipole solution to the PME with first moment $M=\int_0^\infty xu_0(x)\,dx$. If $\hat\delta$ is the constant given by Lemma~\ref{lema-D-V}, then, for all $\delta\in (0,\hat\delta)$,
\[
\limsup_{t\to\infty} t^{\alpha+\frac\beta m}\sup_{0<x<\delta t^\beta}\frac{\big(u(x,t)-D_M(x,t)\big)}{(1+x)^{\frac1 m}}\le0.
\]
\end{prop}

\begin{proof}
Given $\ep>0$, let $T_\ep$ as in Lemma~\ref{lem:control-from-above}, and $\widehat T$ and $a_\ep$ as in Lemma~\ref{lema-D-V}. We take $T\ge \max\{T_\ep,\widehat T\}$ and $a\in(0,a_\ep)$, and then $k_0>1$ large so that the function $V$ defined by~\eqref{eq:V}--\eqref{eq:k} satisfies~\eqref{eq:upper.estimate} and \eqref{eq:epsilon.approximate} for any given $\delta\in(0,\hat \delta)$  for some large $\widehat T_\ep\ge T$.

On the other hand, since
$k(t)\to1$ as $t\to\infty$ and $a\in(0,1)$, using~\eqref{eq:estimate.profile.above} we get
\[
\frac{t^{\alpha+\frac\beta m}V(x,t)}{(1+x)^{\frac1 m}}=\frac{k(t)t^{\frac\beta m}}{(1+x)^{\frac1 m}}F_M\Big(\frac{x+a}{t^\beta}\Big)\le 2 C_m^{\frac1{m-1}}M^{\frac{m+1}{2m^2}}\Big(\frac{x+a}{x+1}\Big)^{\frac1m}\le 2C_m^{\frac1{m-1}}M^{\frac{m+1}{2m^2}}
\]
for all large enough times.

Combining all the estimates mentioned above we finally get, for $0<x<\delta t^\beta$ and all large enough times,
\[
 \frac{t^{\alpha+\frac\beta m}}{(1+x)^{\frac1 m}}\big(u(x,t)-D_M(x,t)\big)\le C_\delta\ep
  \frac{t^{\alpha+\frac\beta m}V(x,t)}{(1+x)^{\frac1 m}}+\ep
  \le (C_\delta 2C_m^{\frac1{m-1}}M^{\frac{m+1}{2m^2}}+1)\ep.
  \]
\end{proof}

\section{Control from below}
\label{sect-subsolution} \setcounter{equation}{0}

We will now deal with the \lq\lq lower'' part of~\eqref{eq:main.result}. The proof is quite similar to that of the \lq\lq upper'' part. However, in this case, subsolutions are only obtained in sets of the form
$$
A_{a,\delta,T}=\{(x,t): a<x<\delta t^\beta, t\ge T\},
$$
and the points $x\in(0,a)$ have to be treated separately.

The subsolution approaching $D_M$ with the right rate as $t$ goes to infinity will have
the form
\begin{equation}
\label{eq:v}
  v(x,t)= c(t) t^{-\alpha}F_M\Big(\frac {x-a}{t^\beta}\Big),\qquad a>0,
\end{equation}
where $c$ is the solution to the Initial Value Problem
\begin{equation}
\label{eq:c}
tc'(t)=\alpha{(c(t)-c^m(t))},\quad t>T,\qquad c(T)=c_0\in(0,1).
\end{equation}
The function $c$ is well defined for $t\ge T$. It is monotone increasing and $c(t)\nearrow 1$ as $t\to\infty$, as desired.

We start by proving that $v$ is a subsolution to the PME in $A_{a,\delta,T}$ if $\delta$ is small and $T$ is big, no matter the value of $a\in(0,1)$.
\begin{lema}
\label{lem:subsolution}
Let $m>1$ and $M>0$, and let $\bar\delta>0$  be as in Lemma~\ref{lem:supersolution}. For all $a\in(0,1)$, $T>0$, and $c_0\in(0,1)$  the function $v$ given by~\eqref{eq:v}--\eqref{eq:c} satisfies
\begin{equation}
\label{eq:subsolution}
v_t-(v^m)_{xx}\le 0\quad\mbox{in } A_{a,\delta,T} \text{ for all }\delta\in(0,\bar\delta).
\end{equation}
\end{lema}
\begin{proof}
Let $\xi=\frac{x-a}{t^\beta}< \frac{x}{t^\beta}$. A computation analogous to the one we did in the proof of Lemma~\ref{lem:supersolution} shows that
\[
\big(v_t-(v^m)_{xx}\big)(x,t)= -
t^{-\alpha-1}(c(t)-c^m(t))\beta\xi F_M'(\xi).
\]
But,
$$
0<\xi=\frac{x-a}{t^\beta}< \frac{x}{t^\beta}<\delta<\bar\delta<\bar\xi M^{(m-1)/(2m)}, \quad (x,t)\in A_{a,\delta,T},
$$
and hence the result follows from~\eqref{eq:positivity.derivative}, since $c(t)<1$ for all times.
\end{proof}

The matching with the outer behavior will require to know that $u$ is positive in some set $A_{\delta,T}$. This is what we prove next.
\begin{lema}
\label{lem:positivity}
Let $u$ be the unique weak solution to \eqref{eq:main}, $M=\int_0^\infty xu_0(x)\,dx$.
Given $\delta\in(0,\xi_M)$, there exists a time $T_{\delta}$ such that $u(x,t)>0$ in~$A_{\delta,T_{\delta}}$.
\end{lema}
\begin{proof}
Since $\delta<\xi_M$, the convergence result~\eqref{eq:KV.result} implies that there is a time $t_\delta$ such that $u(x,t)\ge K\,t^{-\alpha}$ for some $K>0$ if $x\in\big((\delta/2) t^\beta,\delta t^\beta\big)$, $t\ge t_\delta$.

We now use that non-negative solutions to~\eqref{eq:main} have the so-called retention property: if $u(x,\bar t)>0$, then $u(x,t)>0$ for all $t\ge \bar t$. This can be proved in several ways, for instance, using that the application $t\mapsto t^{1/(m-1)}u(x,t)$ is non-decreasing. This monotonicity property follows easily from the estimate $u_t\ge-u/((m-1)t)$, which is proved using comparison arguments; see, for instance, \cite{Vazquez-book}. Hence, $u(x,t)>0$ for $x\in ((\delta/2)t_\delta^\beta,\delta t^\beta)$, $t\ge t_\delta$.

It only remains to prove that $u(x,T_\delta)>0$ if $x\in(0,(\delta/2)T_\delta^\beta)$ for some large enough $T_\delta\ge t_\delta$, since the result will then follow from the retention property. The positivity in this fixed interval is achieved by comparison with a suitable translate  of a source type solution of the PME,
$$
B(x,t;C)=t^{-\frac1{m+1}}\Big(C-\kappa_m|\xi|^2\Big)_+^{\frac1{m-1}},\qquad \xi=x/t^{\frac1{m+1}},\qquad C>0,
$$
where the constant $\kappa_m$ is as in~\eqref{eq:constants}. Such solutions are due to Zel'dovi$\check{\rm c}$ and Kompaneets~\cite{Zeldovic-Kompaneec-1950} and Barenblatt~\cite{Barenblatt-1952}. Indeed, take $x_0\in \big((\delta/2) t_\delta^\beta,\delta t_\delta^\beta\big)$. It is easy to check that if $C>0$ is small enough, then
$$
B(x-x_0,t_\delta;C)=0 \quad\text{if }x\not\in\big((\delta/2) t_\delta^\beta,\delta t_\delta^\beta\big),\qquad \sup_{(\delta/2) t_\delta^\beta\le x\le \delta t_\delta^\beta}B(x-x_0,t_\delta;C) \le K t_\delta^{-\alpha}.
$$
Moreover, $B(x-x_0,t;C)$ is a solution to~\eqref{eq:main} until it touches the boundary $x=0$. This will happen in a finite time $T_\delta\ge t_\delta$. Then, comparison yields that $u(x,t)\ge B(x-x_0,t;C)$ for all $t\in[t_\delta,T_\delta]$, and hence the required positivity.
\end{proof}

We now perform the matching with the outer behavior in order to obtain the control from below.
\begin{lema}
\label{lem:control-from-below}
Let $u$ be the unique weak solution to \eqref{eq:main}, $M=\int_0^\infty xu_0(x)\,dx$,  and $\bar\delta$ as in Lemma~\ref{lem:supersolution}.
Given $\ep>0$, $a\in(0,1)$, and $\delta\in(0,\bar\delta)$, there is a time $T_{\ep,a,\delta}>0$  such that for all  $T\ge T_{\ep,a,\delta}$ there is a value $c_0\in(0,1)$ such that the function $v$ given by~\eqref{eq:v}--\eqref{eq:c} satisfies
\begin{equation}
\label{eq:bound-from-below}
u(x,t)\ge(1-C_\delta\ep)v(x,t),\quad (x,t)\in A_{a,\delta,T}, \qquad C_\delta=1/F_M(\delta).
\end{equation}
\end{lema}

\begin{proof}
Let $\delta\in(0,\bar\delta)$ and $a\in(0,1)$. Note that $\bar\delta<\xi_M$. The convergence result~\eqref{eq:KV.result} implies that, given $\ep>0$,  there exists a value $T_{a,\ep,\delta}\ge\max\{T_{\bar\delta},(a/\delta)^{1/\beta}\}$  such that~\eqref{eq:outer.epsilon.approx} holds with $T_\ep=T_{\ep,a,\delta}$.

Let $T\ge T_{\ep,a,\delta}$. We know from Lemma~\ref{lem:positivity} that there is a constant $\kappa=\kappa(a,\delta,T)$ such that
$u(x,T)\ge \kappa$ if $a<x<\delta T^{\beta}$. Take now $c_0\in(0,1)$ small so that $c_0 T^{-\alpha}F_M(\delta)\le \kappa$. With this choice of $T$ and $c_0$, we define $v$ by~\eqref{eq:v}--\eqref{eq:c}. Using the monotonicity property~\eqref{eq:positivity.derivative}, we get
\[
u(x,T)\ge c_0 T^{-\alpha}F_M\Big(\frac{x-a}{T}\Big)=v(x,T),\qquad a<x<\delta T^\beta.
\]
On the other hand, since $T\ge T_{\ep,a,\delta}$, the convergence result~\eqref{eq:KV.result} together with the selfsimilar form of $D_M$, formula~\eqref{eq:selfsimilar.form.dipole} yield
\[
u(x,t)\ge -\ep t^{-\alpha}+D_M(x,t)=\big(1-F_M(\delta)^{-1}\ep\big)D_M(x,t),\qquad x=\delta t^{\beta},\ t\ge T.
\]

We notice now that $c(t)<1$. Therefore, \eqref{eq:positivity.derivative} implies that $D_M\ge v$ in $A_{\delta, T}$, and we conclude that, for $\ep<F_M(\delta)$ and $C_\delta=1/F_M(\delta)$,
\[
u(x,t)\ge \underbrace{(1-C_\delta\ep)v(x,t)}_{w(x,t)},\qquad x=\delta t^\beta,\ t\ge T.
\]

We now observe that $w$ is a subsolution to the PME in $A_{a,\delta,T}$. Indeed, in that set we have $(v^m)_{xx}(x,t)=-c(t)t^{-(\alpha+1)}\big(\alpha F_M(\xi)+\beta\xi F_M'(\xi)\big)<0$, and hence Lemma~\ref{lem:subsolution} implies that
\[
w_t-(w^m)_{xx}=
(1-C_\delta\ep)\big(v_t-(v^m)_{xx}\big)-\big((1-C_\delta\ep)^m-(1-C_\delta\ep)\big)(v^m)_{xx}\le0\quad\mbox{in }A_{a,\delta,T}.
\]

We finally notice that $w(a,t)=0$ for all $t> T$. Therefore, a comparison argument allows to conclude that~\eqref{eq:bound-from-below} holds.
\end{proof}

The next step is to control the difference between $v$ and $D_M$ for large times.
\begin{lema}
Given $m>1$, $M>0$, and $\ep>0$, let $\hat\delta$ and $a_\ep\in(0,1]$ be as in Lemma~\ref{lema-D-V}. Then, for all $\delta\in(0,\hat\delta)$,  $T>0$, and $a\in(0,a_\ep)$, the function $v$ given by~\eqref{eq:v}--\eqref{eq:c} satisfies
\begin{equation}
\label{eq:difference.D.v}
\frac{t^{\alpha+\frac\beta m}}{(1+x)^{\frac1 m}} \big|D_M(x,t)-v(x,t)\big|<\ep\quad\mbox{in }A_{a,\delta,\widetilde T_\ep}\text{ for some }\widetilde T_\ep\ge T.
\end{equation}
\end{lema}
\begin{proof} Let $x\in(a,\delta t^\beta)$ with $\delta<\hat\delta$. Arguing as in the proof of Lemma~\ref{lema-D-V} we get
\[\begin{aligned}
\frac{t^{\alpha+\frac\beta m}}{(1+x)^{\frac1 m}} & \Big|t^{-\alpha}F_M\Big(\frac{x-a}{t^\beta}\Big)-D_M(x,t)\Big|\le
\frac{t^{\frac\beta m}}{(1+x)^{\frac1 m}}\int_0^1F_M'\Big(\frac{x-sa}{t^\beta}\Big)\frac{x-sa}{t^\beta}\frac a{x-sa}\,ds\\
&\le K\int_0^1(x-sa)^{\frac1m-1}a\,ds\le mK a^{\frac1m}<\frac\ep2
\end{aligned}
\]
if $a<a_\ep:=\min\{(\ep/(2mK))^m,1\}$.

On the other hand, using~\eqref{eq:estimate.profile.above},
\[
\begin{aligned}
\frac{t^{\alpha+\frac\beta m}}{(1+x)^{\frac1 m}} & \Big|t^{-\alpha}F_M\Big(\frac{x-a}{t^\beta}\Big)-v(x,t)\Big|=\frac{t^{\frac\beta m}}{(1+x)^{\frac1 m}}F_M\Big(\frac{x-a}{t^\beta}\Big)|1-c(t)|\\
&\le C_m^{\frac1{m-1}}M^{\frac{m+1}{2m^2}}\left(\frac{x-a}{1+x}\right)^{\frac1m}|1-c(t)|
\le C_m^{\frac1{m-1}}M^{\frac{m+1}{2m^2}}|1-c(t)|<\frac\ep2
\end{aligned}
\]
if $t\ge \widetilde T_\ep$ for some large enough $\widetilde T_\ep\ge T$, since $c(t)\to1$ as $t\to\infty$.

The combination of the above two estimates yields the result.
\end{proof}

We now have  all the tools we need to prove the lower limit. Here a difference arises with respect to the upper limit: we will have to treat separately the limit in sets of the form $(0,a)$ with $a$ small.  This is done using that both $u$ and $D_M$ are small in that set for large times.
\begin{prop}
Let $u$ be the unique weak solution to~\eqref{eq:main}, and $D_M$ be the unique dipole solution to the PME with first moment $M=\int_0^\infty xu_0(x)\,dx$. If $\hat\delta$ is the constant given by Lemma~\ref{lema-D-V}, then, for all $\delta\in (0,\hat\delta)$,
\begin{equation}
\label{eq:lim-inf}
\liminf_{t\to\infty} t^{\alpha+\frac\beta m}\sup_{0<x<\delta t^\beta}\frac{\big(u(x,t)-D_M(x,t)\big)}{(1+x)^{\frac1 m}}\ge0.
\end{equation}
\end{prop}
\begin{proof}
Given $\delta\in(0,\hat\delta)$ and $\ep\in(0,F_M(\delta))$, we choose a small value $a\in(0,a_\ep)$,
with $a_\ep\in(0,1]$ as in Lemma~\ref{lema-D-V}. We will precise how small it has to be later on. We take $T\ge\max\{ T_\ep,T_{\ep,a,\delta}\}$, with $T_\ep$ as in Lemma~\ref{lem:control-from-above}, and $T_{\ep,a,\delta}$ as in Lemma~\ref{lem:control-from-below}, and then $c_0\in(0,1)$ small enough so that the function $v$ defined by~\eqref{eq:v}--\eqref{eq:c} satisfies~\eqref{eq:bound-from-below} and~\eqref{eq:difference.D.v}.


By Lemma~\ref{lem:control-from-above}, we know that there is a value $k_0\ge1$ such that the function $V$ defined by~\eqref{eq:V}--\eqref{eq:k} satisfies~\eqref{eq:upper.estimate}. Besides, since $k(t)\to1$ as $t\to\infty$, there exists a time $\check{T}\ge T$ such that $k(t)\le 2$ for all $t\ge\check{T}$. Therefore, since $C_\delta\ep<1$, using~\eqref{eq:estimate.profile.above} we get
\[\begin{aligned}
\frac{t^{\alpha+\frac\beta m}u(x,t)}{(1+x)^{\frac1 m}}&\le (1+C_\delta\ep)\frac{t^{\alpha+\frac\beta m}V(x,t)}{(1+x)^{\frac1 m}}\\
&\le2(1+C_\delta\ep)\frac{t^{\frac\beta m}}{(1+x)^{\frac1 m}}F_M\Big(\frac{x+a}{t^{\beta}}\Big)\le 4C_m^{\frac1{m-1}}M^{\frac{m+1}{2m^2}} (2a)^{\frac1m}<\frac\ep2,
\end{aligned}
\]
if $a$ is small enough. On the other hand, using again~\eqref{eq:estimate.profile.above},
\[\begin{aligned}
\frac{t^{\alpha+\frac\beta m}D_M(x,t)}{(1+x)^{\frac1 m}}&=
\frac{t^{\frac\beta m}}{(1+x)^{\frac1 m}}F_M\Big(\frac x{t^\beta}\Big)\le C_m^{\frac1{m-1}}M^{\frac{m+1}{2m^2}} a^\frac1m<\frac\ep2.
\end{aligned}
\]
We conclude that
$$
 t^{\alpha+\frac\beta m}\sup_{0<x<a}\frac{\big|u(x,t)-D_M(x,t)\big|}{(1+x)^{\frac1 m}}\le\ep\quad\mbox{if $t$ is large enough}.
$$

We now consider the set $a<x<\delta t^\beta$. Since
$c(t)\to1$ as $t\to\infty$, using~\eqref{eq:estimate.profile.above} we get
\[
\frac{t^{\alpha+\frac\beta m}v(x,t)}{(1+x)^{\frac1 m}}=\frac{c(t)t^{\frac\beta m}}{(1+x)^{\frac1 m}}F_M\Big(\frac{x-a}{t^\beta}\Big)\le 2 C_m^{\frac1{m-1}}M^{\frac{m+1}{2m^2}}\Big(\frac{x-a}{x+1}\Big)^{\frac1m}\le 2C_m^{\frac1{m-1}}M^{\frac{m+1}{2m^2}}
\]
for all large enough times. Combining this estimate with~\eqref{eq:bound-from-below} and~\eqref{eq:difference.D.v}, we finally get, for $a<x<\delta t^\beta$ and all large enough times,
\[
 \frac{t^{\alpha+\frac\beta m}}{(1+x)^{\frac1 m}}\big(u(x,t)-D_M(x,t)\big)\ge- C_\delta\ep
  \frac{t^{\alpha+\frac\beta m}v(x,t)}{(1+x)^{\frac1 m}}-\ep
  \ge -(C_\delta 2C_m^{\frac1{m-1}}M^{\frac{m+1}{2m^2}}+1)\ep.
  \]
\end{proof}



\begin{thebibliography}{CEQW2}

\bibitem{Barenblatt-1952} Barenblatt, G.\,I. \emph{On some unsteady motions of a liquid and gas in a porous medium}. (Russian) Akad. Nauk SSSR. Prikl. Mat. Meh. 16 (1952), no.\,1, 67--78.

\bibitem{Barenblatt-Zeldovich-1957}  Barenblatt, G.\,I.; Zel'dovic, Ya.\,B. \emph{On dipole solutions in problems of non-stationary filtration
of gas under polytropic regime}. (Russian) Prikl. Mat. Mekh. 21 (1957), no.\,5, 718–-720.




\bibitem{CEQW2} Cort\'{a}zar, C.; Elgueta, M.; Quir\'{o}s, F.; Wolanski, N.
\emph{Asymptotic behavior for a nonlocal diffusion
equation on the half line}. Discrete Contin. Dyn. Syst. 35 (2015),  no.\,4, 1391--1407.

\bibitem{Esteban-Vazquez-1988} Esteban, J.\,R.; V\'azquez, J.\,L. \emph{Homogeneous diffusion in R with power-like nonlinear diffusivity}. Arch. Rational Mech. Anal. 103 (1988), no.\,1, 39--80.

\bibitem{Gilding-Peletier-1976} Gilding, B.\,H.; Peletier, L.\,A. \emph{On a class of similarity solutions of the porous media equation}. J. Math. Anal. Appl. 55 (1976), no.\,2, 351--364.

\bibitem{Gilding-Peletier-1977} Gilding, B.\,H.; Peletier, L.\,A. \emph{On a class of similarity solutions of the porous media equation}. II. J. Math. Anal. Appl. 57 (1977), no.\,3, 522--538.

\bibitem{Kamin-Vazquez-1991} Kamin, S.; V\'azquez, J.\,L. \emph{Asymptotic behaviour of solutions of the porous medium equation with changing sign}. SIAM J. Math. Anal. 22 (1991), no.\,1, 34--45.

\bibitem{Vazquez-book} V\'{a}zquez, J.~L.
\lq\lq The porous medium equation. Mathematical theory''.    Oxford
Mathematical Monographs. The Clarendon Press, Oxford University
Press, Oxford, 2007. ISBN: 978-0-19-856903-9.

\bibitem{Zeldovic-Kompaneec-1950} Zel'dovi$\check{\rm c}$, Ya.\,B.; Kompaneets, A.\,S. \emph{On the theory of propagation of heat with the heat conductivity depending upon the temperature}. Collection in honor of the seventieth birthday of academician A.\,F.\,Ioffe, pp.\,61--71. Izdat. Akad. Nauk SSSR, Moscow, 1950.

\end{thebibliography}
\end{document}